\documentclass{amsart}

\usepackage{amsrefs}

\newtheorem{theorem}{Theorem}[section]

\theoremstyle{definition}

\theoremstyle{remark}

\numberwithin{equation}{section}

\begin{document}

\title{On a new class of additive  (splitting) 
operator-difference schemes}

\author{Petr N. Vabishchevich}
\address{Keldysh Institute of Applied Mathematics,
Russian Academy of Sciences, 4 Miusskaya Sq., 
125047 Moscow, Russia}
\email{vabishchevich@gmail.com}

\subjclass[2000]{Primary 65N06, 65M06}

%

\keywords{Evolutionary problems, splitting scheme, 
the stability of operator-difference scheme,
vector additive scheme}

\begin{abstract}
Many applied time-dependent problems are characterized 
by an additive representation of the problem operator.
Additive schemes are constructed using such a splitting 
and associated with the transition to a new time level 
on the basis of the solution of more simple problems for 
the individual operators in the additive
decomposition.
We consider a new class of additive schemes for problems 
with additive representation of the operator at the time 
derivative.  
In this paper we construct and study the vector 
operator-difference schemes, which are characterized 
by a transition 
from one initial the evolution equation to a system 
of such equations.  
\end{abstract}

\maketitle

\section*{Introduction}

For the approximate solution of multidimensional unsteady 
problems of mathematical physics there are widely used 
different classes of additive schemes (splitting schemes)  
\cite{Yanenko,Marchuk,Samarskii}. 
Beginning with the pioneering works \cite{Peaceman, Douglas} 
the most simple 
way to construct additive schemes is in the splitting of the 
problem operator on the sum of two operators with a more 
simple structure --- alternating direction methods, factorized 
schemes, predictor-corrector schemes etc.  
\cite{SamVabAdditive}.

In the more general case of multicomponent splitting, classes 
of unconditionally stable operator-difference schemes are based 
on the concept of summarized approximation.  
In this way, we can construct the classic locally one-dimensional 
schemes (componentwise splitting schemes)
\cite{Marchuk,Samarskii}, additively-averaged locally one-dimensional 
schemes 
\cite{GordMel,SamVabAdditive}.   

A new class of unconditionally stable schemes --- vector 
additive schemes (multicomponent alternating direction method 
schemes) is actively developed 
(see, eg, \cite{Abrashin, VabVectAdd}).  
They belong to a class of full approximation schemes  --- each 
intermediate problem approximates the original one.  
The most simple additive  full approximation schemes are 
based on the principle of regularization of operator-difference 
schemes.
Improving the quality of operator-difference schemes is achieved 
using additive or multiplicative perturbations of operators of 
the scheme  \cite{SamReg}.
Regularized additive schemes for evolutionary equations of the 
first and second order are constructed for equations as well as 
systems of equations   \cite{SamVabReg,VabReg}.
Both the standard schemes of splitting with respect to separate 
directions (locally-onedimensional schemes), splitting with 
respect to physical processes 
and regionally-additive schemes based on domain decomposition 
for constructing parallel algorithms for transient problems of 
mathematical physics  
\cite{VabDDM,Mathew,SamMatVab}.

At present, different classes of additive operator-difference 
schemes for evolutionary equations are constructed via additive 
splitting of the main operator (connected with the solution) 
onto several terms. 
For a number of applications it is interesting to consider 
problems in which the additive representation demonstrates an 
operator at the time derivative.  
In this work, for this new class of evolutionary problems the 
vector additive operator-difference schemes are constructed 
and studied.
The work is organized as follows.  
Section 1 provides a statement of the problem along with 
a simple a priori estimate of the stability for the solutions 
with respect to 
initial data and right-hand side. This estimate is nothing but 
our reference point when considering the vector problem and 
the operator-difference schemes. 
The vector differential problem is considered in Section 2.   
The central part of the work (Section 3) deals with 
the construction and investigation of the stability of 
vector additive schemes.  
Possible generalizations of the results are discussed in Section 4.

\section{Statement of the problem}

Let $H$ be a finite-dimensional Hilbert space, and $A,B,D$  
be linear operators in $H$. 
We consider grid functions $y$  of finite-dimensional real 
Hilbert space 
$H$, for the scalar product and norm in which we use the notations: 
$ (\cdot,\cdot),~~ \|y\| = (y,y)^{1/2}$.
For $\ D = D^* > 0$ we introduce space $H_D$ with scalar 
product $(y,w)_D = (Dy,w)$ and norm $\|y\|_D= (Dy,y)^{1/2}$.

In the Cauchy problem for evolutionary equation of first 
order we search function $y (t) \in H$, which satisfies 
the equation 
\begin{equation}\label{1}
   B \frac {d u} {d t}
   + A u = f(t),
   \quad t > 0
\end{equation}
and the initial condition
\begin{equation}\label{2}
   u(0) = u^0
\end{equation}
at given  $f(t) \in  H$.

We assume that linear operators $A$ and $B$, acting from $H$ 
into $H$ 
($ A: H \to H $, $ B: H \to H$),  are positive, self-adjoint 
and stationary, that is  
\[
  A = A^* > 0, \ \frac{d}{d t} A = A \frac{d}{d t},
  \quad B = B^* > 0, \ \frac{d}{d t} B = B \frac{d}{d t} .
\]
For problem  (\ref{1}), (\ref{2}) we can obtain different 
a priori estimates, which express the stability of 
the solution with respect to 
the initial data and right hand side in different spaces.  
We restrict ourselves to the simplest of them, trying to get 
the same type of estimates for both the scalar and vector problems
as well as for  the solution of both differential and 
difference problems.  

Multiplying scalarly both sides of equation (\ref{1}) 
in $H$ by $u$, we get  
\[
  \frac{1}{2} \frac{d }{d t} (B u, u) +
  (A u, u) = (f, u ) .
\]
For the right hand side we use the estimate 
 \[
  (f, u) \leq 
  (A u, u)  +
  \frac{1}{4} \left (A^{-1} f, f \right ) .
\]
This yields the following a priori estimate for the 
solution of problem (\ref{1}), (\ref{2}):
\begin{equation}\label{3}
   \|u(t)\|^2_B \le \|u^0\|^2_B + 
   \frac{1}{2} \int\limits_0^t \|f(s)\|^2_{A^{-1}} ds,
\end{equation}
which expresses the stability of the solution with respect 
to the initial data and right hand side.  

Standard additive difference schemes are characterized 
by decomposition (splitting) of the operator $A$ onto 
the sum of operators 
of a simpler structure.  
For example, we assume that for operator $A$ we have 
the following additive representation: 
\begin{equation}\label{4}
  A = \sum_{\alpha = 1} ^{p} A_\alpha ,
  \quad  A_\alpha = A_\alpha^* \ge 0,
  \quad \alpha =1,2,...,p .
\end{equation}
Additive difference schemes are based on the basis 
of (\ref{4}), where the problem is decomposed into 
$p$ subproblems.  
The transition from time level $ t^n $  to the next 
level $ t^{n+1} = t^n + \tau$, where 
$\tau > 0 $ is the time step and 
$y^n = y(t^n),~ t^n = n\tau,~ n = 0, 1, ...$, 
is associated with solving problems for individual 
operators $A_\alpha, \alpha = 1,2,...,p$ in additive 
decomposition (\ref{4}).  

The subject of our consideration will be another case.  
In a number of problems the computational complexity 
is not associated with operator $A$, but with 
operator $B$ at the derivatives in time.  
In this case, to decrease the computational complexity 
of problem (\ref{1}), (\ref{2})
we employ the additive representation  
\begin{equation}\label{5}
  B = \sum_{\alpha = 1} ^{p} B_\alpha ,
  \quad  B_\alpha = B^*_\alpha > 0,
  \quad \alpha =1,2,...,p .
\end{equation}
instead of (\ref{4}). 
The transition to a new time level is connected with 
the solution of some auxiliary Cauchy problems for equations  
\[
   B_{\alpha} \frac {d u_{\alpha}} {d t}
   + A u_{\alpha} = f_{\alpha}(t),
   \quad t > 0
   \quad \alpha =1,2,...,p 
\]
with specified appropriate initial conditions. 

\section{Vector problem}

By definition, put ${\bf u}  =  \{u_1, u_2, ..., u_p \}$.
Each individual component is defined as the solution of 
similar problems  
\begin{equation}\label{6}
   \sum_{\beta  = 1} ^{p} B_\beta  \frac {d u_{\beta }} {d t}
   + A u_{\alpha} = f(t),
   \quad t > 0 ,
\end{equation}
\begin{equation}\label{7}
   u_{\alpha}(0) = u^0, 
  \quad \alpha =1,2,...,p .
\end{equation}

Here is the simplest coordinate-wise estimate for the 
stability of the solution.  
Subtracting one equation from another, we get  
\[
   A (u_{\alpha} - u_{\alpha-1} ) = 0,
  \quad \alpha =2,3,...,p .
\]
Taking into account the positivity of operator $A$ this gives  
\[
  u_{\alpha} = u_{\alpha-1} ,
  \quad \alpha =2,3,...,p .
\]
For separate component $u_{\alpha}$ we obtain 
the same equation as for $u$:  
\[
   \sum_{\beta  = 1} ^{p} B_\beta  \frac {d u_{\alpha}} {d t}
   + A u_{\alpha} = f(t),
   \quad t > 0 ,
  \quad \alpha =1,2,...,p .
\]
For the same reason, there are a priori estimates  
\begin{equation}\label{8}
   \|u_{\alpha}(t)\|^2_B \le \|u^0\|^2_B + 
   \frac{1}{2} \int\limits_0^t \|f(s)\|^2_{A^{-1}} ds,
  \quad \alpha =1,2,...,p .
\end{equation}
It follows that  
\[
  u_{\alpha}(t) = u(t) ,
   \quad t > 0 ,
  \quad \alpha =1,2,...,p .
\]
Therefore, as the solution of original problem  
(\ref{1}), (\ref{2})  we can take any component 
of the vector  ${\bf u}(t)$.

For the vector evolutionary problem  we can obtain 
a priori estimates for vector ${\bf u}$, 
considering the problem in Hilbert space 
$\mathbf{H} = H^p$ with the scalar product   
\[
  (\mathbf{u}, \mathbf{v}) =
  \sum_{\alpha=1} ^{p} (u_\alpha, v_\alpha) .
\]
This technique is used, for example, in \cite{SamVabAdditive}  
when considering additive schemes with  splitting (\ref{4}).

We rewrite equations  (\ref{6}) in the form  
\[
   B_\alpha A^{-1} \sum_{\beta  = 1} ^{p} B_\beta  \frac {d u_{\beta }} {d t}
   + B_\alpha u_{\alpha} = \tilde{f}_\alpha(t),
   \quad t > 0 ,
  \quad \alpha =1,2,...,p ,
\]
where  $\tilde{f}_\alpha = B_\alpha A^{-1} f$.
This allows us to write the system of equations in vector form  
\begin{equation}\label{9}
  \mathbf{C} \frac{d \mathbf{u}}{d t} +
  \mathbf{D} \mathbf{u} = \tilde{\mathbf{f}} .
\end{equation}
Operator matrix $\mathbf{C}$ and $\mathbf{D}$ have the form  
\begin{equation}\label{10}
  \mathbf{C} = \{ C_{\alpha \beta } \},
  \quad C_{\alpha \beta } = B_\alpha A^{-1} B_\beta,
\end{equation}
\[
  \mathbf{D} = \{ D_{\alpha \beta } \},
  \quad D_{\alpha \beta } = B_\alpha \delta_{\alpha \beta },
  \quad \alpha, \beta  =1,2,...,p ,
\]
where  $\delta_{\alpha \beta }$ is the Kronecker delta.
Equation (\ref{9}) is supplemented by the initial condition  
\begin{equation}\label{11}
  \mathbf{u}(0) = \mathbf{u}^0 .
\end{equation}

The principal advantage of notation  (\ref{9}) results 
from the fact that 
\[
  \mathbf{C} = \mathbf{C}^* \geq 0,
  \quad \mathbf{D} = \mathbf{D}^* > 0 
\]
in  $\mathbf{H}$. 

Here is a priori estimate for the solution of 
vector problem (\ref{9})--(\ref{11}).
This estimate, on the one hand, is more complicated 
than (\ref{8}) and, on the other hand,
we will use it as the guideline in the consideration 
of the operator-difference schemes.  

Multiplying both sides of (\ref{9}) scalarly  
in $\mathbf{H}$  by  $d \mathbf{u}/ d t$, we get
\begin{equation}\label{12}
  \left (\mathbf{C} 
  \frac{d \mathbf{u}}{d t}, \frac{d \mathbf{u}}{d t} \right ) 
  + \frac{1}{2} \frac{d}{d t} (\mathbf{D} \mathbf{u}, \mathbf{u}) 
  =  \left (\tilde{\mathbf{f}}, \frac{d \mathbf{u}}{d t} \right ) .
\end{equation}  
Taking into account   (\ref{10}), we obtain  
\[
  \left (\mathbf{C} 
  \frac{d \mathbf{u}}{d t}, \frac{d \mathbf{u}}{d t} \right ) 
  = \left (  A^{-1} \sum_{\beta=1}^{p} B_\beta u_{\beta},
  \sum_{\beta=1}^{p} B_\beta u_{\beta} \right ),
\]
and for the right hand side of (\ref{12}) we have 
\begin{equation}\label{13}
  \left (\tilde{\mathbf{f}}, \frac{d \mathbf{u}}{d t} \right ) 
  = \left (  A^{-1} f, \sum_{\beta=1}^{p} B_\beta u_{\beta} \right ) \leq 
  \left (\mathbf{C} 
  \frac{d \mathbf{u}}{d t}, \frac{d \mathbf{u}}{d t} \right ) 
  + \frac{1}{4} \left ( A^{-1} f, f \right ) .
\end{equation}

Similarly (\ref{3}), (\ref{8}),  from  (\ref{12}), (\ref{13})  
it follows the estimate  
\begin{equation}\label{14}
  \| \mathbf{u} \|^2_{\mathbf{D}} \leq 
  \| \mathbf{u}^0 \|^2_{\mathbf{D}} + 
   \frac{1}{2} \int\limits_0^t \|f(s)\|^2_{A^{-1}} ds .
\end{equation}
Taking into account (\ref{10}), we have  
\[
  \| \mathbf{u} \|^2_{\mathbf{D}} = 
  \sum_{\alpha =1}^{p}\left (  B_\alpha u_{\alpha}, u_{\alpha} \right ) .
\]
Thus, estimate  (\ref{14}) can be considered along 
with  (\ref{8}) as the vector analogue of estimate 
 (\ref{3}).  
Taking into account  (\ref{5}), estimate (\ref{12}) 
gives the stability of any individual component of 
vector $\mathbf{u}(t)$.

\section{Additive vector schemes}

Splitting schemes for the approximate solution of 
(\ref{1}), (\ref{2}),  (\ref{5}) will be constructed 
on the basis of usual  
schemes with weights for vector problem  (\ref{6}), (\ref{7}).

The standard two-level scheme with weights for problem 
(\ref{1}), (\ref{2})  has the form 
\begin{equation}\label{15}
  B \frac{y^{n+1} - y^{n}}{\tau } +
  A (\sigma y^{n+1} + (1-\sigma)y^{n} ) = \varphi^n,
  \quad n = 0, 1, ... ,
\end{equation}
where, for example, 
 \[
  \varphi^n = f(\sigma t^{n+1} + (1-\sigma)t^{n} ) ,
\]
and $\sigma$ is a weight parameter (usually  
$0 \leq \sigma \leq 1$).

In  the general theory of operator-difference schemes 
stability developed by A.A. Samarskii  
\cite{Samarskii,SamGul,SamMatVab}, 
there were obtained the exact (unimproved) stability 
criteria for two-level and three-level operator-difference 
schemes in various norms.  
They can be directly used in the study of schemes 
with weights (\ref{15}).  
Here is a typical result.  

\begin{theorem}\label{t-1}
If $\sigma \geq 1/2$, then operator-difference scheme  
(\ref{15}) is absolutely stable in $H_B$ and for 
the difference solution 
the level-wise estimate is valid
\begin{equation}\label{16}
  \|y^{n+1}\|^2_B \leq  \|y^{n}\|^2_B +
  \frac{\tau }{2} \|\varphi^n\|^2_{A^{-1}} .
\end{equation}
\end{theorem}

\begin{proof}
By definition, put 
\[
  y^{\sigma(n)} =
  \sigma y^{n+1} + (1-\sigma)y^{n} =
  \frac{1}{2} ( y^{n+1} +  y^{n}) +
  \tau \left (\sigma - \frac{1}{2} \right )
  \frac{y^{n+1} - y^{n}}{\tau} .
\]
Multiplying scalarly in  $H$ both sides of (\ref{15})  
by  $y^{\sigma(n)}$, we get   
\[
  \frac{1}{2 \tau } (B(y^{n+1} - y^{n}), y^{n+1} + y^{n} ) +
\]
\[
  \tau \left (\sigma - \frac{1}{2} \right )
  \left ( B \frac{y^{n+1} - y^{n}}{\tau }, \frac{y^{n+1} - y^{n}}{\tau } \right )
  + (A y^{\sigma(n)}, y^{\sigma(n)}) 
  = (\varphi^n, y^{\sigma(n)}) .
\]
For the right hand side we use the estimate  
\[
  (\varphi^n, y^{\sigma(n)}) \leq 
  (A y^{\sigma(n)}, y^{\sigma(n)}) +
  \frac{1}{4} (A^{-1} \varphi^n, \varphi^n) .
\]
If $\sigma \geq 1/2$, we obtain desired  estimate 
(\ref{16}) for the stability of the numerical solution 
with respect to
the initial data and right hand side, which is the 
grid analog of  estimate(\ref{3}) for the solution of 
problem  (\ref{1}), (\ref{2}).
This concludes the proof.  
\end{proof}

To solve vector problem (\ref{6}), (\ref{7}) we apply 
the following difference scheme:  
\begin{equation}\label{17}
  B_\alpha \left ( \theta \frac{y_\alpha^{n+1} - y_\alpha^{n}}{\tau} +
  (1-\theta) \frac{y_\alpha^{n} - y_\alpha^{n-1}}{\tau} \right ) +
\end{equation}
\[  
  \sum_{\alpha \neq \beta= 1}^{p} 
  B_\beta \frac{y_\beta^{n} - y_\beta^{n-1}}{\tau } +
  A (\sigma y_\alpha^{n+1} + (1-2\sigma)y_\alpha^{n} + \sigma y_\alpha^{n-1}) 
  = \varphi^n,
\]
\[
  \quad n = 0, 1, ... , 
  \quad \alpha =1,2,...,p .
\]
Unlike (\ref{15})) scheme(\ref{17}) is a three-level 
one and has two weight factors 
$\theta $ and $\sigma $. 

Numerical implementation of scheme (\ref{17}) is associated with 
sequential solving grid problems 
\[
  \left ( \theta B_\alpha^n + \sigma \tau A \right ) y_\alpha^{n+1} = 
  \chi_\alpha^n,
  \quad \alpha =1,2,...,p 
\]
with transition  from time level  $t^n$ to new time level $t^{n+1}$. 
For vector additive scheme  (\ref{17}) it is possible to implement
a parallel organization of computations --- an independent 
calculation of the individual components.  

Using notation (\ref{10}), we write operator-difference 
scheme (\ref{17}) in the vector form  
\begin{equation}\label{18}
  \theta \mathbf{G} 
  \frac{\mathbf{y}^{n+1} - 2\mathbf{y}^{n} + \mathbf{y}^{n-1}}{\tau } + 
\end{equation}
\[
  \mathbf{C} \frac{\mathbf{y}^{n} - \mathbf{y}^{n-1}}{\tau } 
  + \mathbf{D} (\sigma \mathbf{y}^{n+1} + (1-2\sigma)\mathbf{y}^{n} +
  \sigma \mathbf{y}^{n-1} ) = \mathbf{g}^{n}, 
\] 
where  
\[
  \mathbf{G} = \{ G_{\alpha \beta } \},
  \quad G_{\alpha \beta } = B_\alpha A^{-1} B_\alpha \delta_{\alpha \beta },
\]
\[
  \mathbf{g}^{n} = \{ g_\alpha^{n} \},
  \quad g_\alpha^{n} = B_\alpha A^{-1} \varphi^n,
  \quad \alpha, \beta  =1,2,...,p .
\]
Thus, in (\ref{18})  operator $\mathbf{G} = \mathbf{G}^* > 0$.

Taking into account that  
\[
  \frac{\mathbf{y}^{n} - \mathbf{y}^{n-1}}{\tau } =
  \frac{\mathbf{y}^{n+1} - \mathbf{y}^{n-1}}{2 \tau }
  - \frac{\mathbf{y}^{n+1} - 2\mathbf{y}^{n+1} + \mathbf{y}^{n-1}}{2 \tau },
\]
\[
  \sigma \mathbf{y}^{n+1} + (1-2\sigma)\mathbf{y}^{n} +   \sigma \mathbf{y}^{n-1}
  = 
\]
\[
\left ( \sigma - \frac{1}{4} \right ) 
  (\mathbf{y}^{n+1} - 2\mathbf{y}^{n+1} + \mathbf{y}^{n-1} )
  + \frac{1}{4}(\mathbf{y}^{n+1} + 2\mathbf{y}^{n+1} + \mathbf{y}^{n-1} ),
\]
rewrite (\ref{18}) in the form  
\begin{equation}\label{19}
  \mathbf{C} \frac{\mathbf{y}^{n+1} - \mathbf{y}^{n-1}}{2 \tau } 
  + \mathbf{R}
  \frac{\mathbf{y}^{n+1} - 2\mathbf{y}^{n+1} + \mathbf{y}^{n-1}}{\tau } +
\end{equation}
\[
  \frac{1 }{4} \mathbf{D}
  ( \mathbf{y}^{n+1} + 2\mathbf{y}^{n+1} + \mathbf{y}^{n-1})  =
 \mathbf{g}^{n} ,
\]
where  
\[
  \mathbf{R} = 
  \theta \mathbf{G} - \frac{1}{2} \mathbf{C} + 
  \tau \left ( \sigma - \frac{1}{4} \right ) \mathbf{D} .
\]
Let  
\[
  \mathbf{v}^{n} = \frac{1}{2} (\mathbf{y}^{n} + \mathbf{y}^{n-1}),
  \quad \mathbf{w}^{n} = \mathbf{y}^{n} - \mathbf{y}^{n-1}
\]
and rewrite (\ref{19}) in the form 
\begin{equation}\label{20}
  \mathbf{C} \frac{\mathbf{w}^{n+1} + \mathbf{w}^{n}}{2 \tau } 
  + \mathbf{R}
  \frac{\mathbf{w}^{n+1} - \mathbf{w}^{n}}{\tau } +
  \frac{1 }{2} \mathbf{D}
  ( \mathbf{v}^{n+1} + \mathbf{y}^{n})  =
 \mathbf{g}^{n} .
\end{equation}

Multiplying scalarly both sides of (\ref{20}) by  
\[
  2 (\mathbf{v}^{n+1} - \mathbf{v}^{n}) =
  \mathbf{w}^{n+1} + \mathbf{w}^{n} ,
\]
we get the equality
\begin{equation}\label{21}
  \frac{1 }{2 \tau} 
  ( \mathbf{C} (\mathbf{w}^{n+1} + \mathbf{w}^{n}),
    \mathbf{w}^{n+1} + \mathbf{w}^{n}) +
  \frac{1 }{\tau} 
  ( \mathbf{R} (\mathbf{w}^{n+1} - \mathbf{w}^{n}),
    \mathbf{w}^{n+1} + \mathbf{w}^{n}) +
\end{equation}
\[
  ( \mathbf{D} (\mathbf{v}^{n+1} + \mathbf{v}^{n}),
    \mathbf{v}^{n+1} - \mathbf{v}^{n}) =
	(\mathbf{g}^{n}, \mathbf{w}^{n+1} + \mathbf{w}^{n} ) .
\]

Similarly  (\ref{13}),  we have  
\[
  (\mathbf{g}^{n}, \mathbf{w}^{n+1} + \mathbf{w}^{n} ) \leq 
  \frac{1 }{2 \tau} 
  ( \mathbf{C} (\mathbf{w}^{n+1} + \mathbf{w}^{n}) +
  \frac{\tau}{2} 
  (A^{-1} \varphi^n, \varphi^n) .
\]
With this in mind, from (\ref{21}) it follows
\begin{equation}\label{22}
  \mathcal{E}_{n+1} \leq 
  \mathcal{E}_{n} +
  \frac{\tau}{2} 
  (A^{-1} \varphi^n, \varphi^n) ,
\end{equation}
where  
\[
  \mathcal{E}_{n} = 
  ( \mathbf{D} \mathbf{v}^{n}, \mathbf{v}^{n})
  +   \frac{1 }{\tau}
  ( \mathbf{R} \mathbf{w}^{n}, \mathbf{w}^{n}) .
\]

We formulate the conditions under which the value 
of $\mathcal{E}_{n}$ determines the square of 
the norm of the difference solution.  
By virtue of the positivity of operator $\mathbf{D}$ 
it is sufficient to require non-negativity of 
operator $\mathbf{R}$.

For the energy of operators  $\mathbf{C}$ and  
$\mathbf{G}$  holds the following coordinate-wise representation  
\[
  (\mathbf{C} \mathbf{u},\mathbf{u} ) =
  \left (A^{-1} \sum_{\alpha =1}^{p}   B_\alpha u_{\alpha},
  \sum_{\alpha =1}^{p}   B_\alpha u_{\alpha} \right ) ,
\]
\[
  (\mathbf{G} \mathbf{u},\mathbf{u} ) =
\]
Considering
\[
  \left (A^{-1} \sum_{\alpha =1}^{p}   B_\alpha u_{\alpha},
  \sum_{\alpha =1}^{p}   B_\alpha u_{\alpha} \right ) =
  \left ( \sum_{\alpha =1}^{p} 
  \left (A^{-1/2} B_\alpha u_{\alpha} \right )^2,  1 \right ) \leq 
\]
\[
  p \sum_{\alpha =1}^{p} \left ( (A^{-1/2} B_\alpha u_{\alpha})^2, 1 \right ) =
  p \sum_{\alpha =1}^{p} \left (A^{-1} B_\alpha u_{\alpha},
  B_\alpha u_{\alpha} \right )  , 
\]
we get
\[
  \mathbf{C} \leq p \mathbf{G} .
\]
Therefore, at $\sigma \geq 1/4$ and $\theta \geq p/2$  
holds $\mathbf{R} \geq 0$.
We have thus proved the following assertion.  

\begin{theorem}\label{t-2}
If $\sigma \geq 1/4$ and $\theta \geq p/2$, than 
operator $\mathbf{R} \geq 0$ in $\mathbf{H}$, an additive 
vector scheme (\ref{17}) is absolutely stable and for 
the difference solution holds a priori estimate (\ref{22}) with  
\[
  \mathcal{E}_{n} = 
  \left \| \frac{\mathbf{y}^{n} + \mathbf{y}^{n-1}}{2} 
  \right \|_{\mathbf{D}} +
  +   \frac{1 }{\tau}
  \left ( \mathbf{R} (\mathbf{y}^{n} - \mathbf{y}^{n-1}), 
  \mathbf{y}^{n} - \mathbf{y}^{n-1} \right ) .
\]
\end{theorem}

Proved a priori estimate (\ ref (22)) guarantees  
the stability of the difference solution in the half-integer 
time levels (for $\mathbf{v}^{n}$)  and is the difference 
analogue for estimate  (\ref{14}).

\section{Generalizations}

We note some of the key research areas that focus 
on the synthesis and development of the obtained results. 

On the basis of a priori estimate  (\ref{22})  we obtain 
the convergence of the solution of difference problem (\ref{17}) 
to the solution of differential problem  (\ref{1}), (\ref{2}) 
with the first order of $\tau$. 
In the standard way \cite{Samarskii} we consider the problem 
for the truncation error using a particular scheme for finding 
the solution 
at the first time level.  

Instead of (\ref{17}) we can use another additive schemes.  
In the class of vector additive schemes, in particular, 
special attention should be given to the scheme   
\[
  \sum_{\beta= 1}^{\alpha} 
  B_\beta \frac{y_\beta^{n+1} - y_\beta^{n}}{\tau } +
  \sum_{\beta= \alpha + 1}^{p} 
  B_\beta \frac{y_\beta^{n} - y_\beta^{n-1}}{\tau } +
\]
\[
  A (\sigma y_\alpha^{n+1} + (1-2\sigma)y_\alpha^{n} + \sigma y_\alpha^{n-1}) 
  = \varphi^n,
\]
\[
  \quad n = 0, 1, ... , 
  \quad \alpha =1,2,...,p .
\]
In this case,  the time derivative of the several 
components of the vector solution is referred to 
the upper time-level.  
Such vector additive schemes are widely 
used \cite{Abrashin,SamVabMatVect} at usual 
decomposition  (\ref{4}).

Some resources are available when considering more 
general than (\ref{1}), (\ref{2}), (\ref{5}) problems. 
In our study we restricted ourselves to the simplest 
problems, where operators $A, B$ and the components 
of splitting of $B_\alpha, \alpha 1,2,..., p$ are constant 
self-adjoint and positive in finite Hilbert space $H$. 
These restrictions can be removed in some cases, by 
analogy with the theory of additive schemes for problems 
(\ref{1}), (\ref{2}) 
with the usual splitting of (\ref{5}), considering, for 
example, problems with not self-adjoint  operators,  
problem with operator factors \cite{SamMatVab,SamVabAdditive}. 

In terms of generalizing the results, the greatest 
interest is to construct the additive operator-difference 
schemes for solving the Cauchy 
problem for evolutionary equation (\ref{1})  
in the splitting both operator $A$ and operator $B$  --- 
for the problem  (\ref{1}), (\ref{2}), (\ref{4}), (\ref{5}).
In this case the transition to the new time level is based 
on solving a sequence of problems for equations  
\[
   B_{\alpha} \frac {d u_{\alpha}} {d t}
   + A_\alpha  u_{\alpha} = f_{\alpha}(t),
   \quad t > 0
   \quad \alpha =1,2,...,p 
\]
with appropriate initial conditions.


\end{document}